\documentclass[a4paper]{amsart}
\usepackage[latin1]{inputenc}
\usepackage{amssymb}
\usepackage{amsmath}
\usepackage{mathrsfs}
\usepackage{eufrak}
\usepackage{amsthm}
\usepackage{amsfonts}
\usepackage{textcomp}
\usepackage{graphicx}
\usepackage[pdftex]{color}
\usepackage{paralist}
\usepackage[shortlabels]{enumitem}
\usepackage{hyperref}
\usepackage{comment}
\usepackage[arrow, matrix, curve]{xy}

\newtheorem*{corollary*}{Corollary}
\newtheorem{theorem}{Theorem}[section]
\newtheorem*{theorem*}{Theorem}

\newtheorem{corollary}[theorem]{Corollary}
\newtheorem{lemma}[theorem]{Lemma}
\newtheorem{proposition}[theorem]{Proposition}

\newtheorem*{claim*}{Claim}

\theoremstyle{definition}

\newtheorem*{theorem }{Theorem}

\newtheorem{example}[theorem]{Example}

\theoremstyle{remark}

\numberwithin{equation}{theorem}

\makeatletter
\renewcommand*\env@matrix[1][\
arraystretch]{%
  \edef\arraystretch{#1}%
  \hskip -\arraycolsep
  \let\@ifnextchar\new@ifnextchar
  \array{*\c@MaxMatrixCols c}}
\makeatother

\renewcommand{\mod}{\operatorname{mod}}

\newcommand{\Ext}{\operatorname{Ext}}
\newcommand{\Tr}{\operatorname{Tr}}
\newcommand{\CM}{\operatorname{CM}}
\newcommand{\Ref}{\operatorname{Ref}}
\newcommand{\depth}{\operatorname{depth}}
\newcommand{\Per}{\operatorname{Per}}
\newcommand{\Coref}{\operatorname{Coref}}

\newcommand{\Hom}{\operatorname{Hom}}
\newcommand{\Dom}{\operatorname{Dom}}
\newcommand{\ev}{\operatorname{ev}}
\newcommand{\add}{\operatorname{\mathrm{add}}}

\renewcommand{\mod}{\operatorname{mod}}

\begin{document}

\title{On modules $M$ with $\tau(M) \cong \nu \Omega^{d+2}(M)$ for isolated singularities of Krull dimension $d$}
\date{\today}



\author{Ren\'{e} Marczinzik}
\address{Institute of algebra and number theory, University of Stuttgart, Pfaffenwaldring 57, 70569 Stuttgart, Germany}
\email{marczire@mathematik.uni-stuttgart.de}

\begin{abstract}
A classical formula for the Auslander-Reiten translate $\tau$ says that $\tau(M)\cong \nu \Omega^2(M)$ for every indecomposable module $M$ of a selfinjective Artin algebra. We generalise this by showing that for a $2d$-periodic isolated singularity $A$ of Krull dimension $d$, we have for the Auslander-Reiten translate of an indecomposable non-projective Cohen-Macaulay $A$-module $M$, $\tau(M)\cong \nu \Omega^{d+2}(M)$ if and only if $\Ext_A^{d+1}(M,A)=\Ext_A^{d+2}(M,A)=0$.
We give several applications for Artin algebras.
\end{abstract}

\maketitle

\section*{Introduction}

A classical result in Auslander-Reiten theory for Artin algebras is that for every indecomposable non-projective module $X$ over a selfinjective Artin algebra one has $\tau(X) \cong \nu \Omega^2(X)$, see for example proposition 3.7. in \cite{ARS}.
This isomorphism has several important applications for the Auslander-Reiten theory of selfinjective algebras, we refer for example to \cite{DP} and \cite{KZ}.

In section 2 of \cite{RZ} it was recently noted that for a general Artin algebra $A$, every indecomposable non-projective module $X$ satisfies $\tau(X) \cong \nu \Omega^2(X)$ in case $\Ext_A^1(X,A)=\Ext_A^2(X,A)=0$. As a main result of this article, we sharpen their result by showing that also the converse is true.
\begin{theorem*} 
Let $A$ be an Artin algebra with an indecomposable non-projective module $M$.
Then the following are equivalent:
\begin{enumerate}
\item $\Ext_A^1(X,A)=\Ext_A^2(X,A)=0$.
\item $\Tr(X)$ is reflexive.

\item $\tau(X) \cong \nu \Omega^2(X)$.
\end{enumerate}

\end{theorem*}
We actually prove a much more general result for isolated singularities and obtain the previous theorem as a special case by setting $d=0$. We refer to the next section for the relevant definitions.
\begin{theorem*} 
Let $A$ be a $2d-periodic$ isolated singularity of Krull dimension $d$ and $X \in \CM(A)$ indecomposable non-projective.
Then the following are equivalent:
\begin{enumerate}
\item $\Ext_A^{d+1}(X,A)=\Ext_A^{d+2}(X,A)=0$.
\item $\Tr_d(X)$ is reflexive.

\item $\tau(X) \cong \nu \Omega^{d+2}(X)$.
\end{enumerate}

\end{theorem*}
In the final section we give several applications of our main result for Artin algebras.
I thank Jeremy Rickard for allowing me to use his proof in \ref{reflexivecharacterisation}.

\section{A characterisation of $\tau(M) \cong \nu \Omega^{d+2}(M)$}
We assume that $A$ is always a (not necessarily commutative) noetherian ring and modules are finitely generated right modules unless otherwise stated.
We write $(-)^{*}$ short for the functor $Hom_A(-,A)$.
Recall that a module $X$ is called \emph{torsionless} in case the natural evaluation map $\ev_M :M \rightarrow M^{**}$ is injective and it is called \emph{reflexive} in case $\ev_M$ is an isomorphism. 
For a general full subcategory $\mathcal{C}$ of $\mod-A$, we denote by $\underline{\mathcal{C}}$ the subcategory $\mathcal{C}$ modulo projective objects (called the stable category of $\mathcal{C}$) and by $\overline{\mathcal{C}}$ we denote the subcategory $\mathcal{C}$ modulo injective objects (called the costable category of $\mathcal{C}$).
We now define the \emph{Auslander-Bridger transpose} $\Tr(X)$ of a module $X$:
Let 
$$P_1 \xrightarrow{g} P_0 \rightarrow X \rightarrow 0$$
be a projective presentation of $X$. Then define $\Tr(X)$ as the cokernel of the map $g^{*}$.
The following theorem is classical:
\begin{theorem} \label{ABtheorem}
Let $A$ be a noetherian ring.
Then $\Tr$ induces a duality $\underline{mod}-A \rightarrow \underline{mod}-A^{op}$ with $\Tr \Tr \cong id_{\underline{mod}-A}$.
\end{theorem}
\begin{proof}
See proposition 2.6. in \cite{AB}.
\end{proof}

We will need the following equivalent characterisations of reflexive modules. The proof that (2) implies (1) is due to Jeremy Rickard.
\begin{proposition} \label{reflexivecharacterisation}
Let $A$ be a noetherian algebra and $M$ an $A$-module.
The following are equivalent:
\begin{enumerate}
\item $M$ is reflexive.
\item $M \cong M^{**}$.
\item $\Ext_A^1(\Tr(M),A) = \Ext_A^2(\Tr(M),A)=0$.

\end{enumerate}

\end{proposition}

\begin{proof}
The equivalence between (1) and (3) is well known, see formula (0.1) in the introduction of \cite{AB}.
(1) implies (2) by definition. We now show that (2) implies (1):
Assume $M \cong M^{**}$. Then $M$ is a dual and for any dual the canonical evaluation map $M \rightarrow M^{**}$ is a split monomorphism, see for example Proposition 1.1.9. of \cite{C} (note that the book \cite{C} is about commutative rings but the same proof applies to non-commutative rings).
Thus $M$ is a direct summand of $M^{**}=M$ and therefore $M \cong M \oplus N$ for some non-zero module $N$ in case $M$ is not reflexive. We show that this leads to a contradiction and thus $M$ has to be reflexive.
In case $N \neq 0$, there is a strictly increasing chain of submodules of $M$ as follows:
$N < N^2 < N^3 <....$, which contradicts that $M$ is finitely generated and thus noetherian.

\end{proof}

We now give the relevant definitions on isolated singularities, where we closely follow section 3 of \cite{Iya}. We refer to \cite{Iya} for more information and examples.
 Let $R$ be a commutative noetherian complete local ring of Krull dimension $d$ and assume furthermore that $R$ is a Gorenstein ring.
Recall that an $R$-module $X$ is called \emph{Cohen-Macaulay} in case $X=0$ or $\dim(X)=d=\depth(X)$.
 
We always assume that an $R$-algebra $A$ is module-finite and semiperfect so that projective covers exist. We define the full subcategory $\CM(A):= \{ X \in \mod-A | X $ is Cohen-Macaulay as an $R$-module $\}$.
We call $A$ an \emph{$R$-order} in case $A \in \CM(A)$.
We call an $R$-order $A$ an \emph{isolated singularity} in case the Hom-spaces modulo projectives $\underline{\Hom}(X,Y)$ have finite length as $R$-modules for all $X, Y \in \CM(A)$.
For example for $d=0$ the isolated singularities are exactly the Artin algebras.
We assume in the following that $A$ is always an isolated singularity of Krull dimension $d$.
We have a duality $D_d := Hom_R(-,R): \CM(A) \rightarrow \CM(A^{op})$ and use this to define the Nakayama functors as $\nu := D_d (-)^*$ and $\nu^{-1}:= (-)^* D_d$.
The following theorem motivates the definition of the Auslander-Reiten translates for isolated singularities:
\begin{theorem} \label{isosingduality}
Let $A$ be an isolated singularity of Krull dimension $d$.
Then there is a duality $\Omega^d \Tr: \underline{\CM}(A) \rightarrow \underline{\CM}(A^{op})$ with $(\Omega^d \Tr)^2 \cong id_{\underline{\CM}(A)}$.

\end{theorem} 
\begin{proof}
See for example \cite{Iya}, theorem 3.4.

\end{proof}

We set $\Tr_d := \Omega^d \Tr$ and define the \emph{Auslander-Reiten translate} $\tau$ of an isolated singularity of Krull dimension $d$ as $\tau := D_d \Tr_d$ and the inverse Auslander-Reiten translate $\tau^{-1}$ as $\tau^{-1}:= \Tr_d D_d$.
In this text we are interested in a homological characterisation for indecomposable non-projective modules $M \in \CM(A)$ with $\tau(M) \cong \nu \Omega^{d+2}(M)$ for  isolated singularities.

\begin{lemma} \label{Nakfunctor}
Let $A$ be an isolated singularity.
Let $X$ be an indecomposable non-projective $A$-module.
Then $X^{*} \cong  \Omega^2 Tr(X)$.

\end{lemma}
\begin{proof}
Let 
$$P_1 \rightarrow P_0 \rightarrow X \rightarrow 0$$ 
be a minimal projective presentation of $X$ and apply the functor $(-)^{*}$ to it. Since for a projective module $P$ also $P^{*}$ is projective, we obtain by the definition of the Auslander-Bridger transpose the following exact sequence, where $(P_0)^{*}$ and $(P_1)^{*}$ are projective and give a minimal projective presentation of $\Tr(X)$:
$$X^{*} \rightarrow (P_0)^{*} \rightarrow (P_1)^{*} \rightarrow \Tr(X) \rightarrow 0.$$
Thus $X^{*} \cong \Omega^2 \Tr(X)$, since we assumed that $X$ as no projective summands.
\end{proof}

\begin{corollary} \label{newreflexivechar}

Let $A$ be an isolated singularity with an indecomposable non-projective module $X$.
Then $X$ is reflexive if and only if $\Omega^2Tr\Omega^2Tr(X) \cong X$.

\end{corollary}
\begin{proof}
By \ref{reflexivecharacterisation}, $X$ is reflexive if and only if $X \cong X^{**}$. 
Assume first that $X$ is indecomposable non-projective and reflexive.
Note that $X^{*}$ can not be projective, or else also $X^{**} \cong X$ would be projective, which contradicts our assumptions on $X$. 

But by \ref{Nakfunctor}, we have $X^{*} \cong \Omega^2 Tr(X)$ and then $X^{**} \cong \Omega^2 Tr \Omega^2 Tr(X)$.
Now assume that $X$ is indecomposable non-projective with $\Omega^2Tr\Omega^2Tr(X) \cong X$. Then also $\Omega^2(\Tr(X))$ is non-projective and $\Omega^2Tr\Omega^2Tr(X) \cong X^{**} $ forces $X \cong X^{**}$ and thus $X$ is reflexive.

\end{proof}

We call an isolated singularity $A$ of Krull dimension $d$ \emph{$l$-periodic} in case the functor $\Omega^l : \underline{\CM-}(A) \rightarrow \underline{\CM-}(A)$ is isomorphic to the identity functor.

\begin{lemma} \label{helplemma}
Let $A$ be a $2d$-periodic isolated singularity with Krull dimension $d$ and $X \in \CM(A)$ indecomposable non-projective.
Then $\Tr(\Omega^d(X)) \cong \Omega^d(\Tr(X))$ in the stable category.

\end{lemma}
\begin{proof}
We use that $\Omega^d Tr \Omega^d Tr \cong id$ in the stable category by \ref{isosingduality}.
We have for every $X \in \CM(A)$:
$\Omega^{2d}(X) \cong X$. This is equivalent to
$\Omega^d \Tr \Tr \Omega^d(X) \cong X$. 
Note that when for $Y \in \mod-A$: $\Omega^d \Tr(Y)=Z \in \CM(A^{op})$, then also $Y = \Omega^d \Tr(Z) \in \CM(A)$.
This shows that $Y=\Tr \Omega^d(X) \in \CM(A)$. 

We now apply the functor $\Omega^d \Tr$ to $\Omega^d \Tr \Tr \Omega^d(X) \cong X$ and use that $(\Omega^d \Tr)^2 \cong id$ to obtain:
$$(\Omega^d \Tr \Omega^d \Tr )\Tr \Omega^d(X) \cong \Omega^d \Tr(X)$$
which is gives 
$$\Tr \Omega^d(X) \cong \Omega^d \Tr(X).$$
\end{proof}

\begin{theorem} \label{maintheoremorder}
Let $A$ be a $2d-periodic$ isolated singularity of Krull dimension $d$ and $X \in \CM(A)$ indecomposable non-projective.
Then the following are equivalent:
\begin{enumerate}
\item $\Ext_A^{d+1}(X,A)=\Ext_A^{d+2}(X,A)=0$.
\item $\Tr_d(X)$ is reflexive.

\item $\tau(X) \cong \nu \Omega^{d+2}(X)$.
\end{enumerate}

\end{theorem}
\begin{proof}
We have 
$$\Ext_A^{d+1}(X,A)=\Ext_A^{d+2}(X,A)=0 \iff $$
$$\Ext_A^{1}(\Omega^d(X),A)=\Ext_A^{2}(\Omega^d(X),A)=0 \iff $$
$$\Ext_A^{1}(\Tr(\Tr(\Omega^d(X))),A)=\Ext_A^{2}(\Tr(\Tr(\Omega^d(X))),A)=0 \iff $$
$$\Ext_A^{1}(\Tr(\Omega^d(\Tr(X))),A)=\Ext_A^{2}(\Tr(\Omega^d(\Tr(X))),A)=0 \iff $$
$$\Ext_A^{1}(\Tr(\Tr_d(X)),A)=\Ext_A^{2}(\Tr(\Tr_d(X)),A)=0 \iff $$
$\Tr_d(X)$ is reflexive, by our characterision (3) of reflexive modules in \ref{reflexivecharacterisation}. \newline
Here we used that $\Ext_A^{i}(\Tr(\Tr(\Omega^d(X))),A) \cong \Ext_A^{i}(\Tr(\Omega^d(\Tr(X))),A)$ for $i=1,2$ in the stable category since we have $\Tr(\Omega^d(X)) \cong \Omega^d(\Tr(X))$ in the stable category by the previous lemma \ref{helplemma}.
This gives the equivalence between (1) and (2). \newline
Now assume that $\tau(X) \cong \nu \Omega^{2+d}(X)$.
Applying the duality $D_d$ on both sides and using that $(-)^{*} \cong \Omega^2 \Tr(-)$, by \ref{Nakfunctor} this is equivalent to 
$\Tr_d(X) \cong \Omega^2 \Tr \Omega^{2+d}(X)$.
We simplify this using our assumption that $\Tr(\Omega^d(X)) \cong \Omega^d(\Tr(X))$ and $(-)^{**} \cong \Omega^2 \Tr \Omega^2 \Tr$ in the stable category:
$$\Omega^2 \Tr \Omega^{2+d}(X) \cong \Omega^2 \Tr \Omega^2 \Tr \Tr \Omega^d(X) $$
$$ \cong ( \Tr \Omega^d(X))^{**} \cong (\Omega^d \Tr(X))^{**} \cong (\Tr_d(X))^{**}.$$
Thus the isomorphism $\tau(X) \cong \nu \Omega^{2+d}(X)$ is equivalent to $\Tr_d(X) \cong (\Tr_d(X))^{**}$, which is equivalent to $\Tr_d(X)$ being reflexive by (2) of \ref{reflexivecharacterisation}.
\end{proof}

\begin{example}
Let $d \geq 1$.
Let $S=K[[x_0,...,x_d]]$ the formal power series ring in $(d+1)$-variables and $f \in (x_0,...,x_d)$. Set $A=S/(f)$ and assume that $A$ is an isolated singularity.
In this case $A$ is local (and thus semiperfect) and has Krull dimension $d$. $A$ is called an \emph{isolated hypersurface singularity}. We refer to \cite{LW} and \cite{Y} for more information on such algebras and several classification results.
According to theorem 4.9. of \cite{Iya}, we have $\Omega^2 \cong id$ in $\underline{\CM-}(A)$ and thus $A$ is $2d$-periodic and our main result \ref{maintheoremorder} applies to such algebras to give $\tau(M) \cong \nu(M)$ in case $d$ is even and $\tau(M) \cong \nu \Omega^1(M)$ in case $d$ is odd.

\end{example}

\section{Applications for Artin algebras}
We restrict in this section to Artin algebras (that is to the case of Krull dimension $d=0$). We assume that the reader is familiar with the basic representation theory of Artin algebras and refer for example to \cite{ARS}.
We denote the duality of an Artin algebra by $D$. We call a module $M$ \emph{$\tau$-perfect} in case $\tau(M) \cong \nu \Omega^2(M)$ and \emph{$\tau^{-1}$-perfect} in case $\tau^{-1}(M) \cong \nu^{-1} \Omega^{-2}(M)$. $\Per_{\tau}(A)$ is defined as the full subcategory of $\tau$-perfect modules and $\Per_{\tau^{-1}}(A)$ is defined as the full subcategory of $\tau^{-1}$-perfect modules. 
We denote by $\Ref(A)$ the full subcategory of reflexive modules.
Furthermore, we call a module $N$ \emph{coreflexive} in case $D(N)$ is reflexive and we define the subcategory $\Coref(A)$ as the full subcategory of coreflexive $A$-modules.
We state the case $d=0$ of our main result here explicitly due to its importance:
\begin{theorem} \label{maintheorem}
Let $A$ be an Artin algebra with an indecomposable non-projective module $X$.
Then the following are equivalent:
\begin{enumerate}
\item $\Ext_A^{1}(X,A)=\Ext_A^{2}(X,A)=0$.
\item $\Tr(X)$ is reflexive.

\item $X$ is $\tau$-perfect, that is $\tau(X) \cong \nu \Omega^{2}(X)$.
\end{enumerate}

\end{theorem}
We also state the dual of the previous theorem:.

\begin{theorem} \label{dualtheorem}
Let $A$ be an Artin algebra with an indecomposable non-injective module $X$.
Then the following are equivalent:
\begin{enumerate}
\item $\Ext_A^1(D(A),X)=\Ext_A^2(D(A),X)=0$.
\item $\Tr(D(X))$ is reflexive.

\item $X$ is $\tau^{-1}$-perfect, that is $\tau(X) \cong \nu^{-1} \Omega^{-2}(X)$.
\end{enumerate}

\end{theorem}

In selfinjective Artin algebras it is in fact true that for every module $M$ with $\tau(M) \cong \nu \Omega^2(M)$, we even have $\tau(M) \cong \nu \Omega^2(M) \cong \Omega^2 \nu(M)$. But this fails in general as we show in the next example:
\begin{example}
Let $A$ be the Nakayama algebra, given by quiver and relations, with Kupisch series $[2,2,2,1]$ with simple modules numbered from 0 to 3.
Then the simple module $S_0$ satisfies $\tau(S_0) \cong \nu \Omega^2(S_0) \cong S_1$, but $\Omega^2 \nu(S_0) \cong 0$.
\end{example}

The next proposition shows that the classification of $\tau^{-1}$-perfect modules is essentially equivalent to the classification of reflexive modules for general Artin algebras.
\begin{proposition} \label{secondmaintheorem}
Let $A$ be an Artin algebra.
\begin{enumerate}
\item $\tau$ induces an equivalence of categories between $\underline{\Per_{\tau}}(A)$ and $\overline{\Coref}(A)$.
\item $\tau^{-1}$ induces an equivalence of categories between $\overline{\Per_{\tau^{-1}}}(A)$ and $\underline{\Ref}(A)$.
\end{enumerate}
\end{proposition}
\begin{proof}
We only prove (1), since the proof of (2) is dual.
$\tau$ is an equivalence between $\underline{\mod}-A$ and $\overline{\mod}-A$.
Thus when we restricted to the subcategory $\underline{\Per_{\tau}}(A)$, $\tau$ induces an equivalence to the image $\tau(\underline{\Per_{\tau}}(A))$.
By \ref{maintheorem}, $X \in \Per_{\tau}(A)$ implies that $\Tr(X)$ is reflexive and thus $\tau(X)=DTr(X)$ is coreflexive. Thus the image is contained in $\overline{\Coref}(A)$. Now $\tau: \underline{\Per_{\tau}}(A) \rightarrow \overline{\Coref}(A)$ is also dense since when $Y \in \overline{\Coref}(A)$ is given, we have that $\tau^{-1}(Y)$ is $\tau$-perfect. This is true since by \ref{maintheorem}, $\tau^{-1}(Y)$ is $\tau$-perfect if and only if $\Tr(\tau^{-1}(Y))$ is reflexive, which is equivalent to $DTr(\tau^{-1}(Y))=\tau ( \tau^{-1}(Y))=Y$ being coreflexive. 

\end{proof}

We can also obtain a new characterisation of selfinjective Artin algebras from our main theorem:
\begin{corollary}
Let $A$ be an Artin algebra. Then the following are equivalent:
\begin{enumerate}
\item $A$ is selfinjective.
\item Every simple module is $\tau$-perfect.
\end{enumerate}
\end{corollary}
\begin{proof}
Since $\tau \cong \nu \Omega^2$ in a selfinjective Artin algebra, it is clear that (1) implies (2).
Now assume that (2) holds. Then we have $\Ext_A^1(S,A)=0$ for every simple modules $S$ by \ref{maintheorem}.
But by corollary 2.5.4. of \cite{Ben}, we have $\Hom_A(S,\Omega^{-1}(A)) \cong \Ext_A^1(S,A)=0$ for every simple module $S$. Thus the socle of $\Omega^{-1}(A)$ is zero and thus also $\Omega^{-1}(A)$ is zero. This means that $A$ is isomorphic to its injective envelope. But this shows that every projective module is injective and thus $A$ is selfinjective.

\end{proof}

We give two applications of this result to algebras with dominant dimension at least 2 and algebras that are 2-Iwanaga-Gorenstein.
Recall that the \emph{dominant dimension} of a module $M$ with minimal injective coresolution $(I_i)$ is defined as the smallest number $k$ such that $I_k$ is not projective. The dominant dimension of an algebra is by definition the dominant dimension of the regular module.
It is well known that an algebra $A$ has dominant dimension at least one if and only if there is a minimal faithful projective-injective left $A$-module of the form $Af$ for some idempotent $f$. Furthermore, such an algebra has dominant dimension at least two if and only if $A \cong End_{fAf}(Af)$ in addition.
\begin{corollary}
Let $A$ be an algebra with a minimal faithful projective-injective left module $Af$ with dominant dimension at least two. 
There is an equivalence of categories 
$$\mod-fAf/(\add(Af)) \rightarrow \overline{\Per_{\tau^{-1}}(A)}.$$

\end{corollary}
\begin{proof}
Since $A$ has dominant dimension at least two and by a result of Morita in \cite{M} theorem 3.3., an $A$-module $X$ has dominant dimension at least two if and only if $X$ is reflexive.
Now the full subcategory of modules with dominant dimension at least two $\Dom_2$
is equivalent to $\mod-fAf$ and the projective modules are sent to the modules in $\add-Af$ (this is a special case of lemma 3.1. in \cite{APT}) . This means that there is an equivalence $\underline{Ref}(A)=\underline{Dom_2} \cong \mod-fAf/(\add(Af))$. But by \ref{secondmaintheorem}, $\underline{Ref}(A)$ is equivalent to $\overline{\Per_{\tau^{-1}}(A)}$, which finishes the proof.
\end{proof}

\begin{example}
Let $K$ be an infinite field and $V$ an $n$-dimensional vector space.
Let $A$ be the Schur algebra $S(n,r)$ for $n \geq r$. Then $A$ has dominant dimenion at least two and $fAf$ is isomorphic to the symmetric group algebra $KS_r$ and $Af$ to $V^{\otimes n}$, see for example theorem 1.2. in \cite{KSX}.
The previous corollary gives us in this case an equivalence between the module category of the group algebra of the symmetric group modulo the subcategory $\add(V^{\otimes n})$ to the 
costable category of $\tau^{-1}$-perfect modules over the Schur algebra.
\end{example}

Recall that an Artin algebra $A$ is called \emph{$n$-Iwanaga-Gorenstein} in case the left and right injective dimensions of the regular modules $A$ are equal to $n$ for a natural number $n$.
An $A$-module $X$ over an $n$-Iwanga-Gorenstein algebra is called \emph{Gorenstein projective} in case $\Ext_A^i(X,A)=0$ for all $i \geq 1$. All projective modules are Gorenstein projective by definition and one of the main goals of Gorenstein homological algebra is to find a classification of all Gorenstein projective modules for a given Artin algebra.
In the next corollary we give a new characterisation of Gorenstein projective modules for 2-Iwanaga-Gorenstein algebras. Note that this is a large class of algebras that contain for example all cluster-tilted algebras by \cite{KR} and all algebras considered by \cite{GLS}.
\begin{corollary}
Let $A$ be a 2-Iwanaga-Gorenstein algebra.
Then an indecomposable module $X$ is Gorenstein projective if and only if $X$ is $\tau$-perfect, that is $\tau(X) \cong \nu \Omega^2(X)$.

\end{corollary}
\begin{proof}
Since $A$ has injective dimension at most 2, we have that for an arbitrary module $X$: $\Ext_A^i(X,A)=0$ for all $i \geq 1$ if and only if $\Ext_A^1(X,A)=\Ext_A^2(X,A)=0$. But this condition is by \ref{maintheorem} equivalent to $X$ being $\tau$-perfect.

\end{proof}

\begin{example}
Let $A$ be the Nakayama algebra with Kupisch series [3,3,4] and primitive idempotents $e_0, e_1$ and $e_2$. $A$ is 2-Iwanaga-Gorenstein.
The indecomposable non-projective Gorenstein projective $A$-modules are $e_0 A /e_0 J^1$ and $e_1 A/e_1 J^2$.

\end{example}


\begin{thebibliography}{Gus}


\bibitem[AB]{AB} Auslander, M.; Bridger, M.: {\it Stable Module Theory} Memoirs of the American Mathematical Society, No. 94 American Mathematical Society, Providence, R.I. 1969 146 pp. 
\bibitem[APT]{APT} Auslander, M.; Platzeck, M. I.; Todorov, G.: {\it Homological theory of idempotent Ideals} Transactions of the American Mathematical Society, Volume 332, Number 2 , August 1992.
\bibitem[ARS]{ARS} Auslander, M.; Reiten, I.; Smalo, S.: {\it Representation Theory of Artin Algebras} Cambridge Studies in Advanced Mathematics, 36. Cambridge University Press, Cambridge, 1997. xiv+425 pp.

\bibitem[Ben]{Ben} Benson, D. J.: {\it Representations and cohomology I: Basic representation theory of finite groups and associative algebras.} Cambridge Studies in Advanced Mathematics, Volume 30, Cambridge University Press, 1991.


\bibitem[C]{C} Christensen, L. W.: {\it Gorenstein dimensions.} Lecture Notes in Mathematics. 1747. Berlin: Springer. viii, 204 p. (2000).

\bibitem[DP]{DP} Diveris, K.; Purin, M.: {\it Vanishing of self-extensions over symmetric algebras.} Journal of Pure and Applied Algebra 218 (2014) 962-971.

\bibitem[GLS]{GLS} Geiss, C.; Leclerc, B.; Schroer, J.: {\it  Quivers with relations for symmetrizable Cartan matrices I: Foundations.}  Inventiones Mathematicae 209 (2017), 61-158.

\bibitem[Iya]{Iya} Iyama, O.: {\it Auslander-Reiten theory revisited.} Trends in representation theory of algebras and related topics. Proceedings of the 12th international conference on representations of algebras and workshop (ICRA XII), EMS Series of Congress Reports, 349-397 (2008). 
\bibitem[KR]{KR} Keller, B.; Reiten, I.: {\it Cluster-tilted algebras are Gorenstein and stably Calabi-Yau.} Advances in Mathematics Volume 211, Issue 1, 1 May 2007, Pages 123-151.
\bibitem[KZ]{KZ} Kerner, O.; Zacharia, D.: {\it Auslander-Reiten theory for modules of finite complexity over self-injective algebras.} Bull. London Math. Soc. 43 (2011) 44-56.
\bibitem[KSX]{KSX} Koenig, S.; Slungard, I. H.; Xi, C.: {\it Double Centralizer Properties, Dominant Dimension, and Tilting Modules.} Journal of Algebra
Volume 240, Issue 1, 1 June 2001, Pages 393-412.
\bibitem[LW]{LW} Leuschke, G. J.; Wiegand, R.: {\it Cohen-Macaulay Representations.} Mathematical Surveys and Monographs, Volume 181, 2012.
\bibitem[M]{M} Morita, K.: {\it Duality in QF-3 rings.} Mathematische Zeitschrift 108(1969),  237-252.
\bibitem[RZ]{RZ} Ringel, C. M.; Zhang, P. : {\it Gorenstein-projective and semi-Gorenstein-projective modules.} https://arxiv.org/abs/1808.01809.


\bibitem[Y]{Y} Yoshino, Y.: {\it Maximal Cohen-Macaulay Modules over Cohen-Macaulay Rings.} London Mathematical Society Lecture Note Series Book 146, (1990).

\end{thebibliography}
\end{document}